\documentclass[11pt, a4paper,reqno]{amsart}

\usepackage[english]{babel}
\usepackage[T1]{fontenc}  
\usepackage[latin1]{inputenc}   

\usepackage{amsmath,amssymb,amsthm,amsfonts}
\usepackage{url}
\usepackage{nicefrac,enumerate,bbm}

\theoremstyle{plain}
\newtheorem{theorem}{{Theorem}}[section] 
\newtheorem*{theorem*}{{Theorem}}
\newtheorem{proposition}[theorem]{Proposition}
\newtheorem*{proposition*}{Proposition}
\newtheorem{lemma}[theorem]{Lemma}
\newtheorem*{lemma*}{Lemma}

\theoremstyle{definition}
\newtheorem{definition}[theorem]{Definition}
\newtheorem*{definition*}{Definition}

\theoremstyle{remark}
\newtheorem*{remark*}{Remark}
\newtheorem{remark}[theorem]{Remark}

\makeatletter

\@addtoreset{equation}{section}  
\makeatother

\newcommand{\singl}[1]{\left\{ #1 \right\}}		
\newcommand{\R}{\mathbb{R}}

\newcommand{\C}{\mathbb{C}}
\newcommand{\N}{\mathbb{N}}

\newcommand{\tqe}{:}

\renewcommand{\leq}{\leqslant}
\renewcommand{\geq}{\geqslant}
\renewcommand{\bar}[1]{\overline{#1}}
\newcommand{\inv}{^{-1}}
\newcommand {\limt}[2]{\xrightarrow[#1 \to #2]{}}
\newcommand{\abs}[1]{\left\vert #1\right\vert} 
\newcommand{\nr}[1]{\left\Vert #1\right\Vert}         
\newcommand{\innp}[2]{\left< #1 , #2 \right>}         
\newcommand{\Dom}{\Dc}			
\DeclareMathOperator{\Op}{Op}		
\newcommand{\Opw}{\Op_h^w}		
\newcommand{\pppg}[1] {\left< #1 \right>} 	
\newcommand{\bigo}[2]{\mathop{O}\limits_{#1 \to #2}}
\newcommand{\restr}[2]{\left.#1\right|_{#2}}         
\renewcommand{\Re}{\mathop{\rm{Re}}\nolimits}        
\renewcommand{\Im}{\mathop{\rm{Im}}\nolimits}        
\newcommand{\Ran}{\mathop{\rm{Ran}}\nolimits} 
\DeclareMathOperator{\supp}{supp}                    
\newcommand{\1}[1]{\ensuremath{\mathbbm{1}_{#1}}}

\renewcommand{\a}{\alpha}

\newcommand{\g}{\gamma}
\newcommand{\G}{\Gamma}
\renewcommand{\d}{\delta}
\newcommand{\D}{\Delta}
\newcommand{\e}{\varepsilon}

\newcommand{\z}{\zeta}

\renewcommand{\th}{\theta}
\newcommand{\Th}{\Theta}

\renewcommand{\l}{\lambda}

\newcommand{\m}{\mu}
\newcommand{\n}{\nu}
\newcommand{\x}{\xi}

\newcommand{\s}{\sigma}

\renewcommand{\t}{\tau}
\newcommand{\f}{\varphi}
\newcommand{\vf}{\phi}
\newcommand{\h}{\chi}
\newcommand{\p}{\psi}

\renewcommand{\O}{\Omega}

\newcommand{\Dc}{{\mathcal D}}
\newcommand{\Hc}{{\mathcal H}}
\newcommand{\Kc}{{\mathcal K}}
\newcommand{\Lc}{{\mathcal L}}
\newcommand{\Oc}{{\mathcal O}}
\newcommand{\Uc}{{\mathcal U}}
\newcommand{\Vc}{{\mathcal V}}

\newcommand{\gzlde}{{G_{z,h}(\e)}}
\newcommand{\gzlue}{{G_{z,h}'(\e)}}

\newcommand{\gzle}{{G_{z,h}(\e)}}
\newcommand{\kzle}{{G^1 _{z,h}(\e)}}

\newcommand{\alinv}{\pppg {A_h} \inv}   
\newcommand{\alinvs}{\pppg {A_h} ^{-s}} 
\newcommand{\fzle}{{F_{z,h}(\e)}}
\newcommand{\tfzle}{{\tilde F_{z,h}(\e)}}

\newcommand{\rle}{{Q_h(\e)}}

\newcommand{\hol}{{H_0^h}}
\newcommand{\hun}{{H_1^h}}

\begin{document}

\title
	{Limiting absorption principle for the dissipative Helmholtz equation}
\author{Julien Royer}
\address{Laboratoire de mathématiques Jean Leray\\
UMR CNRS 6629\\
Université de Nantes \\
44322 Nantes Cedex 3,  France
}

\email{julien.royer@univ-nantes.fr}

\maketitle

\begin{abstract}
Adapting Mourre's commutator method to the dissipative setting, we prove a limiting absorption principle for a class of abstract dissipative operators. A consequence is the resolvent estimates for the high frequency Helmholtz equation when trapped trajectories meet the set where the imaginary part of the potential is non-zero. We also give the resolvent estimates in Besov spaces.
\end{abstract}

\section{Introduction}

We consider the following Helmholtz equation:
\begin{equation} \label{helmholtz}
\D A(x) + k_0^2 (1-N(x)) A(x) + ik_0  a(x) A(x) = A_0
\end{equation}
This equation modelizes accurately the propagation of the electromagnetic field of a laser in material medium. Here $k_0$ is the wave number of the laser in the vacuum, $N$ and $a$ are smooth nonnegative functions representing the electronic density of material medium and the absorption coefficient of the laser energy by material medium, and $A_0$ is an incident known excitation (see \cite{benamou-al-03}). Note that the laser wave cannot propagate in regions where $N(x) \geq 1$, so it is assumed that $0 \leq N(x) < 1$. An important application of this problem is the design of very high power laser device such as the Laser Méga-Joule in France or the National Ignition Facility in the USA.

The oscillatory behaviour of the solution makes numerical resolution very expensive. Fortunately, the wave length of the laser in the vacuum $2\pi k_0\inv$ is much smaller than the scale of variation of $N$. It is therefore relevant to consider the high frequency limit $k_0 \to \infty$. The simplified model with a constant absorption coefficient has been studied in many papers. This coefficient is either assumed to be positive (see \cite{benamou-al-02, benamou-al-03, wangz06}) in order to be regularizing, or only nonnegative (\cite{wang07}), in which case the outgoing (or incoming) solution has to be chosen for $A$, but in both cases the non-symmetric part of the equation is in the spectral parameter, and what remains is a selfadjoint Schrödinger operator. More precisely we are led to study an equation of the form:
\[
(-h^2 \D + V(x) - (E+i\m_h)) u_h = S_h
\]
where $h \sim k_0\inv$ is a small parameter. 

When the absorption is not assumed to be constant, it has to be in the operator itself and the selfadjoint theory no longer applies. However, the anti-adjoint part is small compared to the selfadjoint Schrödinger operator, so by perturbation theory we can see that some results concerning the selfadjoint part still apply to the perturbed operator.

In this paper we study the limiting absorption principle for the following dissipative Schrödinger operator:
\[
H_h = -h^2 \D + V_1(x) -i \n(h) V_2(x)
\]
on $L^2(\R^n)$, where $V_1$ is a real function, $V_2$ is nonnegative and $\nu : ]0,1] \to ]0,1]$. Note that $\nu(h) = h$ for the Helmholtz equation.\\

In the first section, we prove a uniform and dissipative version of the abstract commutator method introduced by E. Mourre in \cite{mourre81} and developped in many papers (see for instance \cite{perryss81,jensenmp84,jensen85,derezinski-jaksic-01, georgescugm04} and references therein). In particular we see that the anti-adjoint part with fixed sign allows us to weaken the Mourre condition we need to prove uniform estimates and limiting absorption principle on the upper half-plane. On the contrary, the result is not valid on the other side of the real axis.

In section 2 we apply this abstract result to the dissipative Schrödinger operator in the semi-classical setting, following the idea of C. Gérard and A. Martinez (\cite{gerardm88}, see also  \cite{robertt87} for the semi-classical limiting absorption principle). In particular we get uniform estimates of the resolvent $(H_h-z)\inv$ for $h$ small enough, $\Im z > 0$ and $\Re z$ close to $E >0$. In the selfadjoint case, the result is true if and only if $E$ is a non-trapping energy, that is if there is no  bounded classical trajectory for the hamiltonian flow associated to the symbol $p(x,\x) = \x^2 + V_1(x)$ of $H_h$. In the dissipative case, the weakened Mourre condition gives a weaker condition on the classical behaviour: we only have to assume that any bounded trajectory of energy $E$ meets the set where $V_2 > 0$. Note that it is consistent with the selfadjoint result. Section 3 is devoted to prove that this condition is necessary (when $\nu(h) = h$, which is the case we are mainly interested in). To this purpose we use a selfadjoint dilation of the Schrödinger operator and we prove a dissipative Egorov theorem.

Finally, we show that the estimates we have proved in weighted spaces can be extended to estimates in Besov spaces, first for the abstract setting of section 2 and then for the Schrödinger operator.

\section{Commutator method for a family of dissipative operators}

We first recall that an operator $H$ of domain $\Dom(H)$ in the Hilbert space $\Hc$ is said to be dissipative if:
\[
\forall \f \in \Dom (H), \quad \Im \innp {H\f} \f \leq 0
\]

\subsection{Uniform conjugate operators}  \label{conjugate}

Let $(H_h)_{h\in]0,1]}$ be a family of dissipative operators on $\Hc$. We assume that $H_h$ is of the form $H_h = \hol -iV_h$ where $\hol$ is selfadjoint on a domain $\Dom_H$ independant of $h$ and $V_h$ is selfadjoint, nonnegative and uniformly $\hol$-bounded with relative bound less than 1:
\begin{equation}  \label{rel-bound}
\exists a \in[0,1[, \exists b\in\R, \forall h \in ]0,1], \forall \f \in \Dom_H, \quad \nr{V_h \f} \leq a \nr{\hol \f} + b \nr \f
\end{equation}

For any $h \in ]0,1]$ and $\f \in \Dom_H$, write: $\nr{\f}_{\G_h} = \nr {\hol \f} + \nr \f$.
Consider now a family $(A_h)_{h\in]0,1]}$ of selfadjoint operators on a domain $\Dom_A$ independant of $h$, $J\subset \R$ and $(\a_h)_{h\in]0,1]}$ with $\a_h \in ]0,1]$.

\begin{definition} \label{def-conj}
The family $(A_h)$ is said to be uniformly conjugate to $(H_h)$ on $J$ with bounds $(\a_h)$ if the following conditions are satisfied:
\begin{enumerate}[($a$)]
\item
For every $h\in]0,1]$, $\Dom_H\cap \Dom_A$ is a core for $\hol$.
\item
$e^{itA_h}$ maps $\Dom_H$ into itself for any $t\in\R$, $h\in]0,1]$, and:
\begin{equation}
\forall \f \in \Dom_H,\quad \sup_{h\in ]0,1],\abs t \leq 1} \nr{e^{itA_h} \f}_{\G_h} < \infty
\end{equation}
\item
For every $h\in ]0,1]$, the quadratic forms $i[\hol,A_h]$ and $i[V_h,A_h]$ defined on $\Dom_H \cap \Dom_A$ are bounded from below and closable. Moreover, if we denote by $i[\hol,A_h]^0$ and $i[V_h,A_h]^0$ their closures, then $\Dom_H \subset \Dom(i[\hol,A_h]^0)\cap \Dom(i[V_h,A_h]^0)$ and there exists $c\geq 0$ such that for $h\in ]0,1]$ and $\f\in\Dom_H$ we have:
\begin{equation}
\nr{i[\hol,A_h]^0 \f}+\nr{i[V_h,A_h]^0 \f}  \leq c \sqrt {\a_h} \nr{\f}_{\G_h}
\end{equation}
\item
There exists $c \geq 0$ such that for all $\f,\p \in \Dom_H\cap \Dom_A$:
\begin{equation}
\abs{ \innp{i[\hol,A_h]^0 \f}{A_h \p} - \innp{ A_h\f}{ i[\hol,A_h]^0\p} } \leq c \, \a_h \nr{\f} _{\G_h}\nr{\p} _{\G_h}
\end{equation}
and similar estimates hold for the forms $[i[V_h,A_h],A_h]$ and $[V_h,A_h]$.
\item
There exists $C_V \geq 0$ such that for all $h\in ]0,1]$:
\begin{equation}  \label{estim-mourre}
\1 J (\hol) \left( i[\hol,A_h]^0 + C_V V_h \right) \, \1 J (\hol)  \geq \a_h \1 J (\hol)
\end{equation}
where $\1 J$ denotes the characteristic function of $J$ and hence $(\1 I (\hol))_{I\subset \R}$ is the set of spectral projections for the selfadjoint operator $\hol$.
\end{enumerate}
\end{definition}

Let $\C_+ = \{ z\in\C \tqe \Im z > 0\}$ and for $J\subset \R$: $\C_{J,+} = \{ z \in \C_+ \tqe \Re z \in J\}$.

\subsection{Abstract limiting absorption principle}

We first prove a version of the quadratic estimates (see \cite[prop. II.5]{mourre81}) we need in our dissipative case:

\begin{proposition}   \label{propII.5}
Let $T = T_R -iT_I$ be a dissipative operator on $\Hc$ with $T_R$ selfadjoint and $T_I$ nonnegative, selfadjoint and $T_R$-bounded. Then for all $z\in\C_+$ the operator $(T-z)$ has a bounded inverse. Moreover if $B$ is an operator such that $B^* B \leq T_I$ and $Q$ is a bounded selfadjoint operator, then we have:
\begin{equation}
\nr{B (T-z)\inv Q} \leq  \nr{Q (T-z)\inv Q}^{\frac 12}
\end{equation}
\end{proposition}

\begin{proof}
Since $T_R$ is closed and $T_I$ is $T_R$-bounded, $T$ is closed. For $z \in \C_+$ and $\f \in \Dom(H_0)$ we have:
\[
\begin{aligned}
\nr{(T-z)\f} \nr \f 
 \geq \abs{ \Im \innp{(T-z)\f} \f } = \innp{T_I\f}\f + \Im z \nr \f ^2
 \geq \Im z \nr \f ^2 
\end{aligned}
\]
So $(T -z)$ is injective with closed range. We similarly prove that $\nr{(T^* -\bar z)\f} \geq \Im z \nr \f$, so $\Ran(T-z)$ is dense in $\Hc$ and hence equal to $\Hc$, which proves that $(T-z)$ has a bounded inverse. Let $\f \in \Hc$. We compute:
\begin{eqnarray*}
\lefteqn{\nr{B (T-z)\inv Q \f}^2}\\
&& = \innp{B^*B (T-z)\inv Q \f}{ (T-z)\inv Q \f}\\
&& \leq  \innp{T_I (T-z)\inv Q \f}{ (T-z)\inv Q \f} 
 + \Im z   \innp{ (T-z)\inv Q\f}{ (T-z)\inv Q\f}\\
&& \leq \frac  {1} {2i} \innp{Q (T^*-\bar z)\inv [(T^* - \bar z) - (T -  z)] (T-z)\inv Q \f}{ \f}\\
&& \leq \nr{Q (T-z)\inv Q} \nr \f ^2
\end{eqnarray*}
\end{proof}

Let $\pppg \cdot$ denote the function $x \mapsto \sqrt{1 + \abs x^2}$. We can now state and prove the main theorem of this section:

\begin{theorem} \label{th-mourre}
Let $(H_h)_{h\in]0,1]}$ be a family of dissipative operators of the form $H_h = \hol -iV_h$ as in section \ref{conjugate} and $(A_h)_{h \in ]0,1]}$ a conjugate family to $(H_h)$ on the open interval $J \subset \R$ with bounds $(\a_h) _{h\in ]0,1]}$. Then for any closed subinterval $I \subset J$ and all $s > \frac 12$, there exists a constant $c \geq 0$ such that:
\begin{equation} \label{estim-unif}
\forall h \in ]0,1], \forall z \in \C_{I,+}, \quad   \nr{\alinvs (H_h-z)\inv \alinvs} \leq \frac c {\a_h}
\end{equation}
Moreover, we have for all $z,z' \in \C_{I,+}$:
\begin{equation} \label{estimzzp}
\nr{\alinvs \left( (H_h-z)\inv - (H_h-z')\inv \right) \alinvs } \leq c \, {\a_h^{-\frac {4s}{2s+1}}} \abs{z - z'}^{\frac {2s-1}{2s+1}}
\end{equation}
and for $E \in J$ the limit:
\begin{equation} \label{ablim}
\alinvs (H_h -(E +i0)) \inv \alinvs  = \lim_{\m \to 0^+} \alinvs (H_h -(E +i\m)) \inv \alinvs 
\end{equation}
exists in $\Lc(\Hc)$ and is a $\frac {2s-1}{2s+1}$-Hölder continuous function of $E$.
\end{theorem}

\begin{remark}
As in \cite{mourre81}, if we only need resolvent estimates for an operator $H = H_0 -iV$ where $H_0$ is selfadjoint and $V$ is selfadjoint, nonnegative and $H_0$-bounded, we look for a conjugate operator which satifies the same assumptions as in definition \ref{def-conj} with a weaker Mourre condition:
\begin{equation*}
\1 J (H_0) \left( i[H_0,A]^0 + C_V V \right) \, \1 J (H_0)  \geq \a \1 J (H_0) + \1 J (H_0) K \1 J (H_0)
\end{equation*}
where $K$ is a compact operator on $\Hc$. Indeed, for any $E \in J \cap \s_c(H_0)$ (the continuous spectrum of $H_0$) we can find $\d > 0$ such that:
\[
\1 {[E-\d,E+\d]}(H_0) K \1 {[E-\d,E+\d]}(H_0) \geq -\frac \a 2 \1 {[E-\d,E+\d]}(H_0)
\]
hence $A$ is conjugate to $H$ on $[E-\d,E+\d]$ with bound $\frac \a 2$ in the sense of definition \ref{def-conj}.
\end{remark}

The proof of theorem \ref{th-mourre} follows that of the selfadjoint analog:

\begin{proof}
Let $I\subset J$ be a closed interval and $s \in \left]\frac 12 , 1 \right]$ (the conclusions are weaker for $s>1$).
Throughout the proof, $c$ stands for a constant which may change but does not depend on $z \in \C_{I,+}$, $\e \in ]0,1]$ and $h \in ]0,1]$.\\

\noindent {\bf 1.}
Let $\vf \in C_0^\infty(J,[0,1])$ with $\vf= 1$ in a neighborhood of $I$. We set $P_h = \vf(\hol)$ and $P'_h = (1-\vf)(\hol)$. We also define: $\Th_{R,h} = i[\hol,A_h]^0$, $\Th_{I,h} = i[V_h,A_h]^0$, $\Th_h =  \Th_{R,h} -i\Th_{I,h}$ and $\Th_h^V = C_V V_h + \Th_h$, $C_V$ being given by assumption (e).
Then by assumptions (c) and \eqref{rel-bound}, $\Th_h^V$ is $\hol$-bounded and:
\begin{equation} \label{pb-bounded}
\nr{\Th_h P_h} + \nr{P_h \Th_h} \leq c \sqrt {\a_h}  
\end{equation}

The operator $V_h$ is $\hol$-bounded and $P_h \Th_h^V P_h$ is bounded, so for all $h,\e\in ]0,1]$ we can apply proposition \ref{propII.5} with $T_R = \hol - \e P_h \Th_{I,h} P_h$ and $T_I = V_h + \e P_h(C_V V_h + \Th_{R,h})P_h$. Indeed by assumption (e) we have:
\begin{equation} \label{minor-mourre}
\begin{aligned}
0 \leq (\sqrt{\a_h} P_h)^2  = \a_h P_h \1 J(\hol)^2 P_h  \leq   P_h (C_V V_h + \Th_{R,h} ) P_h  
\end{aligned}
\end{equation}
and hence $T_I$ is nonnegative so $\gzle = (H_h - i\e P_h \Th_h^V P_h - z)\inv$ is well-defined for any $z\in\C_+$.

Then we write $\rle = \pppg {A_h}^{-s} \pppg {\e A_h}^{s-1}$ and finally: $\fzle = \rle \gzle \rle$. By functional calculus we have:
\begin{equation} \label{estim-arle}
\nr \rle \leq 1\quad \text{and} \quad \nr{A_h \rle} = \nr{\rle A_h} = \e^{s-1}
\end{equation}
and the second part of proposition \ref{propII.5} with $B = \sqrt {V_h}$ and $Q = \rle$ for all $h,\e\in]0,1]$ gives :
\begin{equation} \label{estim-sqrtV}
\nr{\sqrt {V_h} \gzle \rle} \leq \nr \fzle ^{\frac 12}
\end{equation}

\noindent {\bf 2.}
By \eqref{minor-mourre} and proposition \ref{propII.5} now applied with $B = \sqrt{\a_h} \sqrt \e P_h$, we also have:
\begin{equation} \label{estim-pgr}
\nr{P_h \gzle \rle} \leq  \frac  1 {\sqrt {\a_h} \sqrt \e} \nr\fzle ^{\frac 12}
\end{equation}
On the other hand:
\begin{equation}\label{calcul-pp}
\begin{aligned}
(1+\sqrt {V_h})P'_h \gzle \rle 
& = (1+\sqrt {V_h})P'_h  (\hol-z)\inv(1  +  i(V_h + \e P_h \Th_h^V P_h)\gzle) \rle \\
& = (1+\sqrt {V_h})P'_h  (\hol-z)\inv \rle \\
& \quad + i (1+\sqrt {V_h})P'_h  (\hol-z)\inv V_h  \gzle \rle \\
& \quad + i \e  (1+\sqrt {V_h})P'_h  (\hol-z)\inv  P_h \Th_h P_h\gzle \rle \\
& \quad + i \e C_V (1+\sqrt {V_h})P'_h  (\hol-z)\inv  P_h V_h P_h\gzle \rle 
\end{aligned}
\end{equation}
By functional calculus and \eqref{rel-bound} we have :
\[
\nr{(1+\sqrt{V_h}) P'_h (\hol -z)\inv (1+\sqrt{V_h})} \leq c
\]
With \eqref{estim-sqrtV}, \eqref{pb-bounded} and \eqref{estim-pgr}, this proves that the first three terms of \eqref{calcul-pp} are bounded by $c(1+\nr \fzle^{\frac 12})$. For the last term, since $P_h\sqrt{V_h}$ is uniformly bounded, it only remains to estimate:
\[
\begin{aligned}
\e \nr{\sqrt{V_h} P_h \gzle \rle} 
& \leq \e  \nr{\sqrt{V_h}  \gzle \rle} + \e \nr{\sqrt{V_h} P_h \gzle \rle}\\
& \leq \e  \nr \fzle ^{\frac 12} + \e  \nr{(1+\sqrt{V_h}) P'_h \gzle \rle} 
\end{aligned}
\]


For $\e \in ]0,\e_0]$, $\e _0>0$ small enough, we finally obtain:
\begin{equation} \label{primegzerau}
\nr{P'_h \gzle \rle} + \nr{\sqrt{V_h}P'_h \gzle \rle}  \leq c \left(1 +  \nr \fzle ^{\frac 12} \right)
\end{equation}
Together with \eqref{estim-pgr} this gives:
\begin{equation} \label{estim-gzlerle}
\nr \fzle \leq \nr{\gzle \rle} \leq c\left(1 +\frac {\nr \fzle^{\frac 12}}{\sqrt {\a_h} \sqrt \e}\right)
\end{equation}
and hence:
\begin{equation} \label{estimfze}
\nr \fzle \leq \frac c {\a_h \e}
\end{equation}

Note that by \eqref{rel-bound} we also have:
\begin{equation} \label{estim-H2}
\begin{aligned}
\nr{\hol \gzle \rle }
& \leq  \frac 1 {1-a} \nr{H_h \gzle \rle } + \frac b{1-a} \nr{ \gzle \rle} \\
& \leq c\left(1 +\frac {\nr \fzle^{\frac 12}}{\sqrt {\a_h} \sqrt \e}\right)
\end{aligned}
\end{equation}
while \eqref{estim-sqrtV} and \eqref{primegzerau} give:
\begin{equation} \label{est-sqrtV-P}
\nr{\sqrt {V_h} P \gzle \rle} \leq  c \left(1 +  \nr \fzle ^{\frac 12} \right)
\end{equation}

\noindent {\bf 3.} We now estimate the derivative of $F_{z,h}$ with report to $\e$:
\[
\begin{aligned}
\frac d{d\e} \fzle 
& = i C_V \rle \gzle P_h V_h P_h \gzle \rle \\
& = i \rle \gzle P_h \Th_h P_h \gzle \rle \\
& \quad + \frac {d\rle}{d\e} \gzle \rle + \rle \gzle \frac {d\rle}{d\e}
\end{aligned}
\]
Functional calculus gives:
\[
\nr{\frac {d\rle}{d\e}}\leq c \e ^{s-1}
\]
so the last two terms can be estimated by:
\[
\nr{\frac {d\rle}{d\e} \gzle \rle + \rle \gzle \frac {d\rle}{d\e}} \leq c \e ^{s-1}\left(1 +\frac {\nr \fzle^{\frac 12}}{\sqrt {\a_h} \sqrt \e}\right)
\]
By \eqref{est-sqrtV-P} we have :
\[
\nr{ \rle \gzle P_h V_h P_h \gzle \rle }\leq c (1+\nr \fzle)
\]
and for the second term we replace $P_h \Th_h P_h$ by $\Th_h - P_h \Th_h P'_h - P'_h \Th_h P_h - P'_h \Th_h P'_h$, which gives:
\[
i \rle \gzle P_h \Th_h P_h \gzle \rle = D_1 + D_2 + D_3 + D_4
\]
with:
\[
\begin{aligned}
\nr{D_2}
& = \nr {\rle \gzle P_h  \Th_h P'_h \gzle \rle}\\
& \leq \nr{\rle \gzle} \nr{{P_h  \Th_h }} \nr{ P'_h \gzle \rle}\\
& \leq c \left(1+\frac 1 { \sqrt{\a_h}\sqrt \e} \nr{\fzle} ^{\frac 12} \right) \times c \sqrt {\a_h} \times \left(1+ \nr \fzle ^{\frac 12}\right)\\
& \leq c \left( 1+  \frac 1 {\sqrt \e} \nr \fzle \right)
\end{aligned}
\]
$D_3$ is estimated similarly, while we use \eqref{estim-H2} for $D_4$:
\[
\begin{aligned}
\nr{D_4}
& = \nr {\rle \gzle P'_h \Th_h P'_h \gzle \rle}\\
& \leq \nr {\rle \gzle P'_h}  \nr {\Th_h (\hol -i)\inv P'_h} \nr {(\hol+i) \gzle \rle}\\
& \leq  c \left( 1+  \frac 1 {\sqrt \e} \nr \fzle \right)
\end{aligned}
\]

%
To estimate $D_1$, we are going to use the choice of $\Th_{R,h}$ and $\Th_{I,h}$ as commutators with $H_h$. By proposition II.6 in \cite{mourre81}, $\gzle$ maps $\Dom_A$ into $\Dom_H \cap \Dom_A$, so we can compute, in the sense of quadratic forms on $\Dom_H \cap \Dom_A$:
\begin{eqnarray} \label{terme1}
\nonumber
\lefteqn{\rle \gzle \Th_h \gzle \rle}\\
\nonumber &\hspace{3cm}& 
= i \rle \gzle [H_h,A_h] \gzle \rle \\
\nonumber && 
= i \rle \gzle [H_h - z -i\e P_h \Th_h^V P_h,A_h] \gzle \rle \\
\nonumber && 
\quad - \e  \rle \gzle [P_h \Th_h^V P_h ,A_h] \gzle \rle \\
 && 
= i \rle [A_h ,\gzle] \rle \\
\nonumber && 
\quad - \e  \rle \gzle [P_h \Th_h^V P_h ,A_h] \gzle \rle
\end{eqnarray}

For $\f,\p \in \Dom_H \cap \Dom_A$, we have:
\[
\abs{\innp{\gzle \rle \f}{A_h \rle \f}} \leq c \, \a_h^{-\frac 12} \e ^{s -\frac 32} \nr \fzle^{\frac 12} \nr \f \nr \p 
\]
according to \eqref{estim-gzlerle} and \eqref{estim-arle}. 

By proposition II.6 in \cite{mourre81}, the quadratic form $[P_h \Th_h^V P_h ,A_h]$ has the properties of $[\Th_h^V,A_h]$ given by assumption (d). With \eqref{estim-H2} this proves:
\begin{eqnarray*}
\lefteqn{\e \abs{ \innp{ [P_h \Th_h^V P_h,A_h] \gzle \rle \f}{\gzle ^* \rle \p}}} \\
&\hspace{2cm}& \leq c \, \a_h \, \e  \nr{\gzle \rle \f}_{\G_h} \nr{\gzle \rle \p}_{\G_h} \\
&& \leq c (1+\nr \fzle ) \nr \f \nr \p 
\end{eqnarray*}
So both terms in \eqref{terme1} extend to bounded operators and:
\begin{equation*}
\nr{D_1}  \leq c \, \a_h^{-\frac 12} \e ^{s -\frac 32} \left(1 + \nr \fzle^{\frac 12}\right) + c \left( 1+ \nr\fzle\right)
\end{equation*}
and hence we have proved:
\begin{equation*} 
\nr{\frac d {d\e} \fzle}\leq c + \frac c {\sqrt \e } \nr \fzle + \frac {c\e^{s - \frac 32}}{\sqrt{\a_h}} \nr \fzle ^{\frac 12}
\end{equation*}
which can also be written:
\begin{equation} \label{estim-derivee}
\nr{\frac d {d\e} \a_h \fzle}\leq c + \frac c {\sqrt \e } \nr {\a_h \fzle} + c\e^{s - \frac 32} \nr {\a_h \fzle} ^{\frac 12}
\end{equation}

\noindent {\bf 4.} Using lemma 3.3 in \cite{jensenmp84} with \eqref{estimfze} and \eqref{estim-derivee}, we get that $\fzle$ can be continuously continued for $\e = 0$. Furthermore, the constants in this lemma do not depend on the function but only on the estimates. Since $(\a_h \fzle)$ and its derivative with report to $\e$ are estimated uniformly in $h$, we can conclude that $\a_h F_{z,h}(0)$ is uniformly bounded in $h$, which is exactly \eqref{estim-unif}.\\

\noindent {\bf 5.} Since $\fzle$ is bounded as a function of $\e$, \eqref{estim-derivee} becomes:
\[
\nr{\frac d {d\e} \fzle} \leq c \, \a_h \inv \e ^{s-\frac 32}
\]
and gives:
\begin{equation}  \label{fze-zero}
\nr{\fzle - F_{z,h}(0)} \leq c \, \a_h \inv \e^{s-\frac 12}
\end{equation}
Moreover with \eqref{estim-gzlerle} we get:
\[
\nr{\frac d {dz} \fzle} \leq \nr{\rle \gzle^2 \rle} \leq \nr{ \gzle \rle}^2 \leq \frac {c} {\a_h^2 \e}
\]
and hence for $z,z' \in \C_{I,+}$:
\begin{equation} \label{fzzpe}
\nr{\fzle - F_{z',h}(\e)} \leq  \frac{c\abs{z-z'}}{\a_h^2 \e}
\end{equation}

Take now $z,z' \in \C_{I,+}$ close enough, $h\in]0,1]$ and $\e = \a_h^{-\frac {2}{2s+1}} \abs{z-z'}^{\frac 2 {2s+1}}$. Then \eqref{fze-zero} and \eqref{fzzpe} give \eqref{estimzzp}. 

In particular, for $E \in I$ the map $\m \mapsto F_{E+i\m,h}(0)$ has a limit for $\m \to 0^+$, and taking the limit $\m \to 0$ in \eqref{estimzzp} with $z = E +i\m$ and $z'= E'+i\m$ shows that the limit is Hölder-continuous with report to $E$ and finishes the proof.
\end{proof}

\begin{remark}
We added the uniform estimate on $[V_h,A_h]$ in assumptions (d) because we had to put $V_h$ in the $\e$-term of $\gzle$ in order to use the weak Mourre estimate \eqref{estim-mourre}. But this assumption is useless if we can take $C_V = 0$ in \eqref{estim-mourre}.\footnote{ Initially, estimate $\nr{V_h(H_0^h+i)^{-1}}
=O(\sqrt{\alpha _h})$ was also required in assumption (c). I thank Th. Jecko for pointing out that this can be avoided.}
\end{remark}

\section{Application to the dissipative Helmholtz equation}  \label{sec-appl-schr}

In this section we apply the abstract Mourre theory to the dissipative Schrödinger operator. Let $V_1 \in C^\infty(\R^n,\R)$ with:
\begin{equation} \label{estimV}
\forall \a \in \N^n, \forall x \in \R^n, \quad \abs{\partial ^\a V_1 (x)} \leq C_\a \pppg x ^{-\rho - \abs \a}
\end{equation}
for some $\rho > 0$ and $C_\a \geq 0$. Let $V_2 \in C^\infty(\R^n,\R)$ nonnegative, with bounded derivatives (up to any order) and:
\begin{equation}
V_2(x) \limt {\abs x} \infty 0
\end{equation}
  
We consider on $L^2(\R^n)$ the operator:
\[
 H_h = -h^2 \D + V_1 - i \nu(h) V_2
\]
where $\n(h) \in ]0,1]$. We denote by $\hun = -h^2 \D +V_1(x)$ the selfadjoint part of $H_h$, $\tilde \nu (h) = \min(1,\nu(h)/h)$ and:
\[
\Oc = \{ (x,\x) \in \R^{2n} \tqe V_2(x) > 0 \}
\]
We also write $\Opw (a)$ for the Weyl-quantization of a symbol $a$ (see \cite{robert,martinez,evansz}):
\[
\Opw (a) u (x) = \frac 1 {(2\pi h)^n} \int_{\R^n} \int_{\R^n} e^{\frac ih \innp{x-y} \x} a\left(\frac{x+y}2,\x \right) u(y)\, dy \, d\x
\]

\subsection{Hamiltonian flow}

Let $p : (x,\x) \mapsto \x^2 + V_1(x)$ be the symbol of $\hun$, and $(x_0,\x_0) \mapsto \vf^t(x_0,\x_0) = (\bar x(t,x_0,\x_0),\bar \x(t,x_0,\x_0)) \in \R^{2n}$ the corresponding hamiltonian flow:
\[
\begin{cases}
\partial_t \bar x(t,x_0,\x_0) = 2 \bar \x(t,x_0,\x_0) \\
\partial_t \bar \x(t,x_0,\x_0) = -\nabla V_1 (\bar x(t,x_0,\x_0)) \\
\vf^0(x_0,\x_0)  = (x_0,\x_0)
\end{cases}
\]

For $I \subset \R$ we introduce:
\[
\begin{aligned}
\O_b(I) &= \{ w \in p\inv(I) \tqe \{ \bar x(t,w)\}_{t\in\R} \text { is bounded} \} \\
\O_\infty^\pm(I) &= \{ w \in p\inv(I) \tqe 
\abs{ \bar x(t,w)} \limt t {\pm \infty} \infty \} \\
\end{aligned}
\]

We recall a few basic facts about this flow:

\begin{proposition} \label{prop-flot}
\begin{enumerate}[(i)]
\item For $a \in C^\infty(\R^{2n})$ we have $\partial_t( a\circ \vf^t) =  \{p,a \circ \vf^t \}$ where $\{ \cdot , \cdot \}$ is the Poisson bracket.
\item If $I \subset \R_+^*$ is closed, there exists $R_0(I) > 0$ such that for any $R \geq R_0(I)$, a trajectory of energy in $I$ which leaves $B_x(R)$ (in the future or in the past) cannot come back.
\item If $I \subset \R_+^*$, $p\inv (I) = \O_b(I) \cup \O_\infty^+(I) \cup  \O_\infty^-(I)$.
\item If $I \subset \R_+^*$ is closed, then $\O_b(I)$ is compact in $\R^{2n}$.
\item If $I \subset \R_+^*$ is open, then $\O_\infty^+(I)$ and $\O_\infty^-(I)$ are open.
\end{enumerate}
\end{proposition}

\subsection{Limiting absorption principle for the dissipative Schrödinger operator}

\begin{theorem}  \label{th-schrodinger}
Let $E >0$ and $s >\frac 12$. If all bounded trajectories of energy $E$ meet $\Oc$, then there exists $c \geq 0$, $h_0>0$ and $I = [E-\d,E + \d]$, $\d > 0$, such that:
\begin{enumerate}[(i)]
\item
For all $h \in ]0,h_0]$:
\begin{equation} \label{unif-schro}
\sup_{z \in \C_{I,+}} \nr{ \pppg x ^{-s}  (H_h-z)\inv \pppg x ^{-s}} \leq \frac c {h\tilde \nu (h)}
\end{equation}
\item
For all $h \in ]0,h_0]$ and $z,z' \in \C_{I,+}$:
\begin{equation}  \label{estim-diff}
\nr{\pppg x ^{-s} \left( (H_h-z)\inv - (H_h - z') \inv \right) \pppg x ^{-s}} \leq  c \, (h\tilde \nu (h))^{-\frac{4s}{2s+1}} \abs {z-z'}^{\frac{2s-1}{2s+1}}
\end{equation}
\item
For $\l \in I$ and $h\in]0,h_0]$ the limit:
\begin{equation} \label{limite}
\pppg x ^{-s}  (H_h-(\l+i0))\inv \pppg x ^{-s} = \lim_{\m \to 0^+} \pppg x ^{-s}  (H_h-(\l+i\m))\inv \pppg x ^{-s}
\end{equation}
exists in $\Lc(L^2(\R^n))$ and is a $\frac{2s-1}{2s+1}$-Hölder continuous function of $\l$.
\end{enumerate}
\end{theorem}

\begin{remark}
This condition that a damping perturbation of the Schrödinger operator allows to weaken a non-trapping condition already appears in \cite{aloui-khenissi-07} where dispersive estimates are obtained for the Schrödinger operator on an exterior domain.
\end{remark}

\begin{remark}
We are mainly interested in the cases $\nu (h) = h$ ($\tilde \nu (h) = 1$), as mentionned in the introduction, and $\nu (h) = h^2$ ($\tilde \nu (h) = h$) which appears in the study of the high energy limit for the Schrödinger operator $-\D - i V_2 - z$, $\Re z \gg 1$ (see \cite[§1.2]{aloui-khenissi-07}).
\end{remark}

\begin{remark} \label{remarque-non-captif}
If $E$ is a non-trapping energy, we have the usual estimate in $O(h\inv)$, no matter how small the anti-adjoint part is.
\end{remark}

The proof of theorem \ref{th-schrodinger} follows that of the selfadjoint case given in \cite{gerardm88}: we find a conjugate family of operators using the quantization of an escape function and then we check that this operators can be replaced by $\pppg x$ in the results of theorem \ref{th-mourre}. The only difference is that we need to prove a weaker Mourre estimate so we are allowed to consider a function which is not an escape function where $V_2$ is not zero. Let us denote:
\begin{equation} \label{def-ah}
A_h = \frac 1 2 (x.hD + hD.x)
\end{equation}
the generator of dilations.

\begin{proposition}
For any $r \in C_0^\infty(\R^{2n},\R)$ the operators $\tilde \nu (h) F_h = \tilde \nu(h)(A_h + \Opw (r))$ are selfadjoint and satisfy assumptions (a) to (d) for a conjugate operator to $H_h$.
\end{proposition}

The proof of this proposition is not really changed by the imaginary part of $V$, so we omit it. The important assumption is the Mourre estimate (e), for which we need to chose $r$ more carefully:

\begin{proposition}
Assume that every bounded trajectory of energy $E$ goes through $\Oc$, then there exists $\e >0$ and $r \in C_0^\infty(\R^{2n},\R)$ such that $\tilde \nu (h) F_h = \tilde \nu (h) (A_h + \Opw(r))$ is conjugate to $H_h$ on $J = [E-\e,E+\e]$ with bounds $c_0 h \tilde \nu(h)$, where $c_0 >0$.
\end{proposition}

\begin{proof}
\noindent {\bf 1.} We first remark that the assumption on bounded trajectories can be extended to a neighborhood of $E$: there exists $\e \in \left]0,E/{12}\right]$ such that any bounded trajectory of energy in $[E- 3\e, E + 3 \e]$ meets $\Oc$. Indeed, assume that for any $n \in \N$ we can find $w_n$ in the compact set $\O_b([E/2,2E])$ such that $p(w_n) \to E$ and $w_n \notin \Oc^\vf$ where:
\[
\Oc^\vf = \bigcup_{t\in\R} \vf^{-t} (\Oc)
\]
Maybe after extracting a subsequence we can assume that $w_n \to w \in \O_b([E/2,2E])$. As $p$ is continuous, we have $p(w) = E$, and hence $w \in\Oc^\vf$ which is open. This gives a contradiction. We set $J= ]E-\e,E+\e[$, $J_2=]E-2\e,E+2\e[$ and $J_3 = ]E-3\e,E+ 3\e[$.\\

\noindent {\bf 2.}
Let $R\geq R_0(\bar J _3)$ (given in proposition \ref{prop-flot}) so large that $\O_b(\bar J_3) \subset B_x(R)$, where $B_x(R) = \{ (x,\x) \in \R^{2n} \tqe \abs x < R\}$, and:
\[
\abs {2V_1(x) + x.\nabla V_1(x)} \leq \frac E2 \quad  \text{when }\abs x \geq R 
\]
Let $b \in C^\infty(\R^n)$ equal to $x.\x$ outside $B_x(R+1)$ and zero in a neighborhood of $\bar B_x(R)$. Then, if $p(x,\x) \in J_3$ and $\abs x \geq R+1$ we have:
\begin{equation} \label{minor1}
\{ p , b \} (x,\x) = 2 p(x,\x) - 2 V_1(x) - x.\nabla V_1(x) \geq E
\end{equation}
and $\{p,b\} = 0$ in $B_x(R)$.\\

\noindent {\bf 3.}
Let $w \in \O_b(\bar {J_3})$ and $T_w \in\R$ such that $\vf^{T_w}(w) \in \Oc$. As $\vf^{T_w}$ is continuous, we can find $\g_w > 0$ and an open neighborhood $\Vc_w$ of $w$ in $\R^{2n}$ such that for any $z \in \Vc_w$ we have $\vf^{T_w}(z) \in \Oc_{\g_w}$ where $\Oc_\g$ stands for $\{ (x,\x) \in \R^{2n} \tqe V_2 (x) > \g\}$. Let $\Uc_w$ be another neighborhood of $w$ with $\bar {\Uc_w} \subset \Vc_w$, $g_w \in C_0^\infty(\R^{2n},[0,1])$ be supported in $\Vc_w$ and equal to 1 on $\Uc_w$, and $f \in C^\infty(\R^{2n})$ defined for $z \in \R^{2n}$ by:
\[
f_w(z) = \int_0^{T_w} g_w(\vf^{-t}(z))\, dt
\]
$f_w$ has been chosen to satisfy:
\[
\begin{aligned}
\{p,f_w\} (z)
& = \int_0^{T_w} \{p, g_w \circ \vf^{-t}\} (z) \, dt = - \int_0^{T_w} \frac d {dt} g_w(\vf^{-t}(z)) \, dt\\
& = g_w(z) - g_w(\vf^{-{T_w}}(z)) 
\end{aligned}
\]
The first term is supported in $\Vc_w$, nonnegative and equal to 1 on $\Uc_w$ while the support of the second is in $\vf^{T_w} (\Vc_w) \subset \Oc_{\g_w}$. In particular $\{p,f_w\}$ is compactly supported, nonnegative outside $\Oc_{\g_w}$ and equal to 1 in $\Uc_w \setminus \Oc_{\g_w}$.

As $\O_b(\bar {J_3})$ is compact, we can find $w_1,\dots,w_N \in \O_b(\bar {J_3})$ for some $N\in\N$ such that $\O_b(\bar {J_3}) \subset \Uc := \cup_{j=1}^N \Uc_{w_j}$. Let $ \g = \min_{1\leq j \leq N} \g_{w_j}$ and $f = \sum_{j=1}^N f_{w_j}$. Then $\{p, f\}$ is compactly supported, nonnegative outside $\Oc_\g$ and greater than or equal to 1 in $\Uc \setminus \Oc_\g$.\\

\noindent {\bf 4.}
We can find a constant $C_V \geq 0$ such that $\{ p , f \} + C_V V_2 \geq 1$ on $\Oc_\g$, so that $\{ p,f\}  + C_V V_2$ is nonnegative on $\R^{2n}$ and at least 1 on $\Uc$.\\

\noindent {\bf 5.}
Let:
\[
\Uc_\pm = \O_\infty^\pm({J_3}) \cap B_x(R+2)   
\]
We have:
\[
\Uc_+ \cup \Uc_- \cup  \Uc \cup   p\inv(\R \setminus \bar J_2) \cup \left(\R^{2n} \setminus \bar B_x(R+1)\right) = \R^{2n} 
\]
Considering a partition of unity for this open cover of $\R^{2n}$ provides two functions $g_\pm \in C_0^\infty(\R^{2n},[0,1])$ supported in $\Uc_\pm$ such that $g_\infty = g_+ + g_-$ is equal to 1 in a neighborhood of the compact set:
\[
K_\infty = p\inv(\bar J_2) \cap \bar B_x(R+1) \setminus \Uc
\]

There exists $T \geq 0$ such that for any $w \in \R^{2n}$ we can find a neighborhood $\Vc$ of $w$ and $\t_{\pm} \geq 0$ such that for any $v \in \Vc$ and $t\geq 0$ we have:
\[
0 \leq g_\pm (\vf^{\pm t}(v)) \leq \1 {[\t_{\pm},T+\t_{\pm} ]}(t)
\]
%
%
As a consequence the functions:
\[
f_\pm = \mp \int_0^{+\infty} (g_\pm \circ \vf^{\pm t}) \, dt
\]
are well-defined, bounded (by $T$) and $C^\infty$ on $\R^{2n}$. The same calculation as for $f$ shows that $\{p,f_\pm\} = g_\pm \geq 0$. Hence we can find a constant $C_\infty \geq 0$ such that for $f_\infty = f_+ + f_-$ we have:
\begin{equation}  \label{minor-fin}
\{ p, b + C_\infty f_\infty \} \geq  E \quad \text{on } K_\infty
\end{equation}
and we already know that $\{p , b + C_\infty f_\infty\} \geq \{p,b\}$ is nonnegative on $p\inv(\bar J_2) \setminus K_\infty$.

\noindent {\bf 6.}
Let $\z \in C_0^\infty(\R^n)$ equal to 1 on $B(R+2)$. Since we can replace $\z$ by $x \mapsto \z(\m x)$ with $\m$ small enough, we can assume that:
\[
\nr {C_\infty f \{ p,\z\} } _{L^\infty(p\inv(J_2))} \leq  2 C_\infty T \sup_{p\inv(J_2)} \abs{\x . \nabla \z(x)} \leq \frac E 2 
\]
With \eqref{minor1} and \eqref{minor-fin} this gives:
\begin{equation}\label{minor-symb}
\{p,b+ \z f_\infty \} \geq \frac E 2 \quad \text{on } p\inv(J_2) \setminus \Uc
\end{equation}
and $\{p,b+ \z f_\infty \}$ is still nonnegative on $p\inv(J_2)$ since $\nabla \z = 0$ on $\Uc$.
Taking $r(x,\x) =  x.\x - b(x,\x) +   C_\infty \z f_\infty +    f$ then $r \in C_0^\infty(\R^{2n})$ and:
\[
\{ p , x.\x + r \} + C_V  V_2 \geq 2 c_0  \quad \text{on } p\inv(J_2) \text{ with } 2 c_0 = \min \left( 1 , \frac E 2 \right)
\]

\noindent {\bf 7.}
Let $F_h = A_h + \Opw (r) = \Opw (x.\x + r )$. The principal symbol of the operator $ih\inv [H_ {1,h},F_h]$ is $\{p, x.\x + r\}$. Let $\h \in C_0^\infty (\R)$ supported in $J_2$ and equal to 1 on $J$. By \cite{robert} or \cite{helfferr83} the operator $\h(\hun)$ is a pseudo-differential operator of principal symbol $\h \circ p$. As a consequence the principal symbol of the operator:
\[
\frac i h \h(\hun) [\hun,F_h]\h(\hun) + C_V V_2 - 2c_0 \h(\hun)^2
\]
is nonnegative, so by G\aa rding inequality (see theorem 4.27 in \cite{evansz}) there is $C\geq 0$ such that, after multiplication by $h \tilde \nu(h)$:
\[
\h(\hun) i\left[\hun,\tilde \nu (h) F_h \right]\h(\hun) + h \tilde \nu(h)  C_V  V_2 \geq  2h\tilde \nu (h) c_0 \h(\hun)^2 - C h^2 \tilde \nu(h)
\]
Taking $h$ small enough and multiplying by $\1 J (\hun)$ on both sides give:
\[
\1 J (\hun) \left( i\left[\hun,\tilde \nu (h) F_h \right] +  h \tilde \nu (h) C_V V_2 \right) \1 J (\hun)  \geq h \tilde \nu (h) c_0 \1 J (\hun)^2
\]
Then $(\nu(h) - h \tilde \nu(h)) \1 J (\hun) C_V V_2\1 J (\hun) \geq 0$ so we have:
\[
\1 J (\hun) \left( i\left[\hun, \tilde \nu (h) F_h \right] + C_V  \nu (h) V_2 \right) \1 J (\hun)  \geq  h \tilde \nu (h) c_0 \1 J (\hun)^2
\]
which is the Mourre estimate we need.
Note that if $E$ is non-trapping we can take $f=0$, $C_V = 0$, and use the estimate:
\[
\1 J (\hun)  i[\hun,F_h]  \1 J (\hun)  \geq  h c_0 \1 J (\hun)^2
\]
even if $h \geq \nu(h)$, which justifies remark \ref{remarque-non-captif}.
\end{proof}

This proposition shows that for any closed subinterval $I$ of $J$ theorem \ref{th-schrodinger} is true with $\pppg {\tilde \nu(h) F_h}^{-s}$ instead of $\pppg x ^{-s}$. The operator $\pppg {\tilde \nu(h) F_h}^s \pppg {\tilde \nu (h) A_h}^{-s}$ is bounded uniformly in $h$ (this is true for $s=0$ and $s=1$ hence for any $s \in [0,1]$ by complex interpolation), so conclusions of theorem \ref{th-schrodinger} are valid with $\pppg {\tilde \nu (h) A_h}^{-s}$. Now write:
\[
(H_h-z)\inv = (H_h-i)\inv - (z-i) (H_h-i)^{-2} + (z-i)^2 (H_h-i)\inv (H_h-z)\inv (H_h-i)\inv
\]
Since $\pppg {\tilde \nu(h) A_h}^{s}(H_h-i)\inv \pppg x ^{-s}$ is uniformly bounded (see \cite[lemma 8.2] {perryss81}), this gives:
\begin{eqnarray}
 \nonumber\lefteqn{ \nr{\pppg x ^{-s} (H_h-z)\inv \pppg x ^{-s}}}\\
\label{devhi} && \leq c + c\nr{\pppg x ^{-s} (H_h-i)\inv (H_h-z)\inv (H_h-i)\inv \pppg x ^{-s}}\\
\nonumber && \leq c + c\nr {\pppg x ^{-s} (H_h-i)\inv \pppg {\tilde \nu(h) A_h}^{s}} \nr{\pppg {\tilde \nu(h) A_h}^{-s} (H_h-z)\inv \pppg {\tilde \nu(h) A_h}^{-s}}\\
\nonumber && \hspace{1.5cm} \times \nr{\pppg {\tilde \nu(h) A_h}^{s}(H_h-i)\inv \pppg x ^{-s}} \\
\nonumber && \leq \frac c {h\tilde \nu(h)}
\end{eqnarray}
where $c$ does not depend on $z \in \C_{I,+}$ with $\Im z \leq 1$. This is \eqref{unif-schro}. Then \eqref{estim-diff} and hence existence of the limit \eqref{limite} follow similarly.

\section{Necessity of the condition on trapped trajectories}

We consider in this section the operator $H_h = -h^2 \D + V_1 -ihV_2$ we introduced to study the Helmholtz equation. We prove that our assumption that every bounded trajectory of energy $E$ should meet the open set $\Oc$ is actually necessary in order to have the uniform estimates and the limiting absorption principle as in theorem \ref{th-schrodinger}. When $V_2 = 0$, this is proved in \cite{wang87}.

\begin{theorem} \label{th-reciproque}
Assume that for some $s \in \left] \frac 12 , \frac {1+\rho}2\right[$ ($\rho > 0$ given by \eqref{estimV}), there exists $\e,h_0 > 0$ such that the limit:
\[
\pppg x ^{-s} (H_h-(\l+i0))\inv \pppg x ^{-s}
\]
exists for all $\l \in J = ]E-\e,E+\e[$ and $h\in ]0,h_0]$ with the estimates:
\[
\nr {\pppg x ^{-s} (H_h-z)\inv \pppg x ^{-s}} \leq \frac c h
\]
uniformly in $z \in \C_{J,+}$ and $h\in ]0,h_0]$, then every bounded trajectory of energy $E$ goes through $\Oc$.
\end{theorem}

To prove this theorem we use the contraction semigroup generated by $H_h$ (given by Hille-Yosida theorem, see for instance theorem 3.5 in \cite{engel2}):
\[
U_h(t) = e^{-\frac {it} h H_h}, \quad t \geq 0
\]
We first need a dissipative version of the Egorov theorem. Let $q \in C^\infty(\R_+ \times \R^{2n})$ be defined by:
\[
q(t,w) = \exp \left( -2\int_{0}^t V_2(\vf^s(w))\, ds \right)
\]
(where $V_2(x,\x)$ means $V_2(x)$ for $(x,\x) \in \R^{2n}$).

\begin{theorem} \label{th-egorov}
Let $a \in C^\infty(\R^{2n})$ be a symbol whose derivatives are bounded (in $L^\infty(\R^n)$). Then for all $t \geq 0$ we have:
\begin{equation} \label{egorov}
U_h(t)^* \Opw(a) U_h(t) = \Opw \left( (a\circ \vf^t) q(t) \right) + h R(t,h)
\end{equation}
where $R$ is bounded in $\Lc(L^2(\R^n))$ uniformly in $h \in ]0,1]$ and $t$ in a compact subset of $\R_+$. 
\end{theorem}

\begin{remark}
More precisely, we prove that there exists a family $(b(\t,h))_{\t \geq 0}$ of classical symbols with bounded derivatives (uniformly for $\t$ in a compact subset of $\R_+$) such that for all $t \geq 0$:
\begin{equation} \label{explicit-reste}
R(t,h) = \int_{0}^t U_h(\t)^* \Opw(b(\t,h)) U_h(\t)\, d\t
\end{equation}
\end{remark}

\begin{remark} \label{rem-egorov}
If we replace one of the $U_h(t)$ by $U_1^h(t) = e^{-\frac{it}h \hun}$ in the left-hand side of \eqref{egorov} then we have to replace $q$ by 
\begin{equation}\label{def-q1}
q_1 : (x,\x) \mapsto \exp \left( -\int_{0}^t V_2(\vf^s(w))\, ds \right)
\end{equation}
in the right-hand side (with the two occurences of $U_h(t)$ replaced by $U_1^h(t)$ and $q$ replaced by 1, theorem \ref{th-egorov} is just the usual Egorov theorem). 
\end{remark}

\begin{proof} 
We follow the proof of the usual Egorov theorem (see for instance \cite[§~IV.4]{robert}). Let $t \geq 0$. For $\t \in [0,t]$ and $w \in \R^{2n}$ write: 
\[
\tilde a(\t,w) = a(\vf^{t-\t}(w)) \exp (S(\t,w)) \quad \text{where} \quad S(\t,w) = -2\int_\t^t V_2(\vf^{s-\t}(w)) \, ds
\]
and: 
\[
B_h( \t) = U_h(\t)^* \Opw(\tilde a(\t)) U_h(\t)
\]
so that the estimate we have to prove is: $B_h(t) - B_h(0) = O(h)$ in $\Lc(L^2(\R^n))$.
We have:
\[
\begin{aligned}
\partial_\t \tilde a(\t)
& = -\{ p , a\circ \vf^{t-\t} \} \exp (S(\t)) + 2 \left( V_2 + \int_\t^t \{ p, V_2 \circ \vf^{s-\t} \}  \, ds \right) \tilde a(\t) \\
& = -\{ p , a\circ \vf^{t-\t} \} \exp (S(\t)) + 2 V_2 \tilde a(\t) -  \left \{ p,S(\t) \right \}    \tilde a(\t) \\
& = 2 V_2 \tilde a(\t) - \{ p ,\tilde  a(\t) \} 
\end{aligned}
\]
The function $\t \mapsto B_h(\t)$ is of class $C^1$ in the weak sense and:
\[
B'_h(\t) = U_h(\t)^* \tilde B_h (\t) U_h(\t)
\]
with:
\[
\begin{aligned}
\tilde B_h(\t)
& = \frac ih [\hun,\Opw(\tilde a(\t))] -V_2 \Opw(\tilde a(\t)) - \Opw(\tilde a(\t)) V_2 + \Opw (\partial_\t\tilde a(\t))\\
& = \Opw (c(\t,h)) 
\end{aligned}
\]
for some classical symbol $c(\t,h) = \sum_{j\in\N} h^j c_j(\t)$, and in particular:
\[
c_0(\t) = \{p,\tilde a(\t)\} - V_2 \tilde a(\t) - \tilde a(\t)V_2 + \partial_\t \tilde a(\t) = 0
\]
Setting $b = h\inv c$ we get \eqref{egorov}-\eqref{explicit-reste}, in the weak sense and hence in $\Lc(L^2(\R^n))$.
\end{proof}

\begin{proposition} \label{prop-rec-1}
Assume that the assumptions of theorem \ref{th-reciproque} are satisfied. Then for any $\h \in C_0^\infty(\R)$ supported in $J$ there exists $c\geq 0$ such that for all $h\in ]0,h_0]$ and $z \in \C_+$ we have:
\begin{equation} \label{estim-h-res}
\nr{\pppg x ^{-s} \h(\hun) (H_h-z)\inv \h(\hun) \pppg x ^{-s}} \leq \frac c {h}
\end{equation}
\end{proposition}

\begin{remark}
We have similar estimates for $(H_h^* - \bar z)\inv$.
\end{remark}

%
%
%
%

\begin{proof}
First, we can find $c\geq 0$ such that estimate \eqref{estim-h-res} holds for $z \in \C_{J,+}$ by assumption and uniform boundedness of $\pppg x^{\mp s} \h(\hun) \pppg x^{\pm s}$ with report to $h$ (note that this last statement holds for $s=0$ by functional calculus and for $s=1$, we use the fact that $\h(\hun)$ is a pseudo-differential operator whose symbol has bounded derivatives and $[x,\Opw(b)] = -ih\Opw(\partial_\x b)$; then the claim follows for any $s \in [0,1]$ by complex interpolation).

Then, there exists $\d > 0$ such that for all $z \in \C_{\R\setminus J,+}$ we have $d(z,\supp\h) \geq \d$. As a consequence, the operator $\h(\hun) (\hun - z)\inv$ is bounded uniformly in $z \in \C_{\R\setminus  J,+}$ and $h \in ]0,h_0]$. Hence, using twice the resolvent equation, we can write:
\begin{eqnarray*}
\lefteqn{\nr{\h(\hun) (H_h-z)\inv \h(\hun)}}\\
&& \leq c + h^2 \nr{\h (\hun) (\hun-z)\inv V_2 (H_h-z)\inv V_2 (\hun-z)\inv \h(\hun)}\\
&& \leq c \left (1 + h \nr{\sqrt {h V_2} (H_h-z)\inv \sqrt{h V_2}} \right) \\
&& \leq c 
\end{eqnarray*}
where the last step is given by proposition \ref{propII.5} applied with $T = H_h = \hun -ih V_2$ and $B = Q = \sqrt{hV_2}$.
\end{proof}
%
%
%

%
%
%
%
%
%
%

\begin{proposition} \label{prop-U}
Assume that the assumptions of theoreme \ref{th-reciproque} are satisfied. Then for any $\h \in C_0^\infty(\R)$ supported in J there exists $C_\h \geq 0$ such that for all $\p\in L^2(\R^n)$ and $h\in]0,h_0]$ we have:
\begin{equation} \label{maj-int}
\int_{0}^{+\infty} \nr{\pppg x ^{-s} \h(\hun) U_h(t)\p}^2 \, dt \leq C_\h \nr\p^2
\end{equation}
\end{proposition}


\begin{proof}
Let $K_h$ be the selfadjoint dilation of $H_h$ on the Hilbert space $\Kc \supset L^2(\R^n)$ given in appendix \ref{sec-dil}. Let $P$ be the orthogonal projection of $\Kc$ on $L^2(\R^n)$ and $A_h = \pppg x ^{-s} \h(\hun)\in \Lc(\Kc)$, where operators on $L^2(\R^n)$ are extended by 0 on $L^2(\R^n)^\bot \subset \Kc$. Let $\f = (\f_0 , \f_\bot) \in \Kc  = L^2(\R^n) \oplus L^2(\R^n)^\bot$. For $z \in \C_+$ we have:
\begin{eqnarray*}
\lefteqn{\abs { \innp {A_h^* \f}{\left((K_h-z)\inv - (K_h-\bar z)\inv\right) A_h^* \f}_\Kc}}\\
&& = \abs { \innp {\f_0}{\pppg x^{-s} \h(\hun) \left((H_h-z)\inv - (H_h^*-\bar z)\inv\right) \h(\hun) \pppg x^{-s}   \f_0}_{L^2(\R^n)}}\\
&& \leq \frac{2c} h \nr{\f_0}^2_{L^2(\R^n)}  \leq \frac{2c} h \nr{\f}^2_\Kc 
\end{eqnarray*}
where $c$ is given by proposition \ref{prop-rec-1}. The same applies if $\Im z < 0$, so by theorem XIII.25 in \cite{rs4}, where $h$-dependance has to be checked for our semiclassical setting, this proves that $A_h$ is $K_h$-smooth and:
\begin{equation} \label{estim-K}
\sup_{h\in ]0,h_0]} \sup_{\nr\f = 1} \int_{\R} \nr{A_h e^{-\frac {it}h K_h }\f }_{\Lc(\Kc)}^2 \, dt < \infty
\end{equation}
But for $\p \in L^2(\R^n)$ (which we identify with $(\p,0)\in\Kc$), $h \in ]0,h_0]$ and $t\geq 0$ we have:
\[
\nr{\pppg x ^{-s} \h(\hun)   U_h(t)  \p }_{\Lc(L^2(\R^n))} = \nr{\pppg x ^{-s} \h(\hun) P  e^{-\frac {it}h K_h } P \p }_{\Lc(\Kc)} =  \nr{A_h e^{-\frac {it}h K_h } \p}_{\Lc(\Kc)} 
\]
so \eqref{estim-K} gives \eqref{maj-int}.
\end{proof}

\begin{proposition}
Let $T \geq 0$ and $\h \in C_0^\infty$ as in proposition \ref{prop-rec-1}. There exists $h_T >0$ and $C'_\h \geq 0$ such that for any $\p \in L^2(\R^n)$ and $h \in ]0,h_T]$ we have:
\begin{equation}
 \int_0^T \nr{ \pppg x ^{-s} \h(\hun) U_1^h(t)  Q_h(T) \p }^2 \, dt \leq C'_\h \nr \p ^2
\end{equation}
where $Q_h(T) = \Opw (q_1(T))$, $q_1$ being defined in \eqref{def-q1}.
\end{proposition}

\begin{proof}
According to Egorov theorem  applied with the symbol $a(x,\x) = 1$ we have:
\[
U_1^h(-t) U_h(t) = Q_h(t) + hR(t,h)
\]
where $R$ is bounded in $\Lc(L^2(\R^n))$ uniformly for $h \in ]0,1]$ and $t\in[0,T]$. On the other hand, writing $Q_h(t,T) = \Opw(q_1(t,T))$ with $q_1(t,T) = \left(e^{-\int_t^T V_2 \circ \vf^\t \, d\t}\right)$ for $t \in [0,T]$ we have by theorem 5.1 in \cite{evansz}:
\[
\nr{ Q_h(t,T)} \leq C + \bigo h 0 (\sqrt h) \quad \text{and} \quad Q_h(T) = Q_h(t,T) Q_h(t) + \bigo h 0 (h)
\]
where $C$ does not depend on $t,T$ and $h$, and the sizes of the remainders in $\Lc(L^2(\R^n))$ depend on $T$ but can be estimated uniformly on $t \in [0,T]$. Then if $\nr \p = 1$ we have: 
\begin{eqnarray*}
\lefteqn{\int_{0}^T \nr{\pppg x ^{-s} \h(\hun) U_1^h(t) Q_h(T)\p}^2\,dt}\\
&& \leq \int_0^T \nr{\pppg x^{-s} \h(\hun) U_1^h(t) Q_h(t,T) Q_h(t) \p}^2 \, dt + \bigo h 0 (h)\\
&& \leq \int_0^T \nr{\pppg x^{-s} \h(\hun) Q(2t,T+t)  U_1^h(t)  Q_h(t) \p}^2 \, dt + \bigo h 0 (h)\\
&& \leq \int_0^T \nr{Q(2t,T+t)  \pppg x^{-s} \h(\hun) U_h(t)\p}^2 \, dt + \bigo h 0 (h)\\
&& \leq \left( C + \bigo h 0 (\sqrt h) \right) \int_0^T \nr{\pppg x^{-s} \h(\hun)   U_h(t) \p}^2 \, dt + \bigo h 0 ( h)\\
&& \leq C C_\h + \bigo h 0 (\sqrt h)
\end{eqnarray*}
where $C_\h$ is given by proposition \ref{prop-U}. The remainder is uniformly bounded in $\p$ so we can chose $h_T > 0$ small enough to make it less than 1 and the result follows with $C'_\h = CC_\h + 1$.
\end{proof}

We can now prove theorem \ref{th-reciproque} as in \cite{wang87}:

\begin{proof}[Proof of theorem \ref{th-reciproque}]
Let $A_h$ be  the generator of dilations defined in \eqref{def-ah} and $\h,\f,\p \in C_0^\infty(\R)$ supported in $J$ such that $\h(E) = 1$ and $\h(\l) = \l \f(\l) \p (\l)$ for all $\l \in \R$. We have:
\[
\hun U_1^h(T) = \frac 1 {2T} \left( [A_h,U_1^h(T)] + \int_0^T U_1^h(T-t) W(x)  U_1^h(t) \, dt \right) 
\]
where $W(x) = - 2 V_1(x) - x.\nabla V_1(x)$ and hence there exists $c\geq 0$ such that for all $T \geq 0$ and $h \in ]0, h_T]$ ($h_T > 0$ depends on $T$):
\begin{eqnarray}
\label{est-1surT} \lefteqn{\nr{\pppg x ^{-s}   Q_h(T) \h(\hun) U_1^h(T) Q_h(T)  \pppg x ^{-s}}}\\
\nonumber&\hspace{1cm}& = \nr{\pppg x ^{-s}   Q_h(T) \f(\hun) \hun U_1^h(T) \p(\hun) Q_h(T)  \pppg x ^{-s}}\\
\nonumber && \leq \frac c {T}(1 + \nr {F_h(T)})
\end{eqnarray}
where:
\[
F_h(T) =  \int_0^T  \pppg x ^{-s} Q_h(T)\f(\hun) U_1^h(T-t) W(x) U_1^h(t) \p(\hun)Q_h(T)  \pppg x ^{-s} \,dt
\]
Indeed, we have $\nr{Q_h(T)} \leq C + O(\sqrt h)$, hence for $h \in ]0,h_T]$ with $h_T >0$ small enough we have $\nr{Q_h(T)} \leq 2C$. Furthermore $A_h$ is uniformly $\hun$-bounded, so we have:
\[
\nr{ \pppg x ^{-s}  Q_h(T) \f(\hun) [A_h,U_1^h(T)] \p(\hun) Q_h(T)  \pppg x ^{-s}} \leq c
\]
uniformly in $T\geq 0$ and $h \in ]0,h_T]$.

Let us now chose $\th \in C_0^\infty(\R^n)$ with support in $B(0,2)$ and equal to 1 on $B(0,1)$, and define $W_1(T,x) = W(x) \th(x/T)$, $W_2(T,x) = W(x) - W_1(T,x)$ and $F_j^h(T)$ with the same expression as $F_h(T)$ with $W$ replaced by $W_j$ ($j=1,2$). As $W$ decays like $V_1$ (see \eqref{estimV}), there exists $c\geq 0$ such that for all $T \geq 0$ and $h\in]0,h_T]$ we have $\nr {F_2^h(T)} \leq c T^{1-\rho }$. To estimate $F_1^h$ we compute, for $\nr f _{L^2(\R^n)} = \nr g _{L^2(\R^n)} = 1$:
\[
\begin{aligned}
\abs{\innp{F_1^h(T) f} g}
& \leq \int_0^T \nr{\pppg x^{-s} U_1^h(t) \p(\hun)Q_h(T)  \pppg x ^{-s} f} \nr{ \pppg x ^{2s} W_1(t,x)} \\
& \hspace{2cm} \times  \nr{\pppg x^{-s} U_1^h(T-t) \f(\hun)Q_h(T) \pppg x ^{-s} g}\, dt\\
& \leq c T^{2s-\rho } \int_0^T \nr{\pppg x^{-s} \p(\hun) U_1^h(t) Q_h(T)  \pppg x ^{-s} f}^2\,dt\\
& \hspace{2cm} \times \int_{0}^T \nr{\pppg x^{-s} \f(\hun)  U_1^h(T-t) Q_h(T) \pppg x ^{-s} g}^2\, dt\\
& \leq c\, T^{2s-\rho } 
\end{aligned}
\]
where $c$ is independant of $T \geq 0$ and $h \in ]0,h_T]$. Finally we have:
\begin{equation} \label{estim-fh}
\nr{F_h(T)} \leq c T^{1 - \d}
\end{equation}
with $\d = \min( 1+\rho - 2s,\rho) >0$ and $c \geq 0$ independant of $T \geq 0$ and $h \in ]0,h_T]$.

Let $(z,\z) \in \O_b(E)$ (if $\O_b(E)$ is empty then there is nothing to prove) and $T \geq 0$. Let $W_h(z,\z) = \exp\left(ih^{-\frac 12} (\z.x - z.D) \right)$ (see \cite[§ 3.1] {wang85}) and:
\[
G_h(T) =  W_h(z,\z)^* \big<{h^{\frac 12} x} \big> ^{-s}  R_h(T) \h(P_1^h) V_h(T) R_h(T) \big<{h^{\frac 12} x}\big> ^{-s} V_h(-T)  W_h(z,\z)
\]
where $P_1^h = -h\D + V_1(h^{\frac 12}x)$, $V_h(T) = \exp\left( -\frac {iT}h P_1^h\right)$ and $R_h(T) = q_1(T)^w(h^{\frac12} x , h^{\frac 12} \x)$. These three operators are conjugate to $\hun$, $U_h(T)$ and $Q_h(T)$ by the unitary transformation $f \mapsto \left(x \mapsto h^{\frac n 4} f( h^{\frac 12}x)\right)$, so for $T\geq 0$ and $h \in ]0,h_T]$ we have by \eqref{est-1surT} and \eqref{estim-fh}:
\[
\begin{aligned}
\nr{G_h(T)} 
& = \nr{ \big<h^{\frac 12} x\big>^{-s}  R_h(T) \h(P_1^h) V_h(T) R_h(T)  \big<h^{\frac 12} x\big> ^{-s}}\\
& = \nr{ \pppg x ^{-s}    Q_h(T)  \h(\hun) U_1^h(T) Q_h(T) \pppg x ^{-s}}\\
& \leq c T^{-\d}
\end{aligned}
\]
where $c$ does not depend on $T$ and $h \in ]0,h_T]$. On the other hand, 
using \cite[lemma 3.1]{wang85} and \cite[theorem 4.2]{wang86} we have:
\[
\begin{aligned}
G(T) 
& = \big< { h^{\frac 12} x +z} \big>^{-s} q_1(T)^w(h^{\frac 12}x +z,h^{\frac 12}D + \z) \, (\h\circ p)^w(h^{\frac 12}x +z,h^{\frac 12}D + \z)\\
& \quad \times  W_h(z,\z)^*  V_h(T) q_1(T)^w(h^{\frac 12}x,h^{\frac 12}D)  \big<{h^{\frac 12} x}\big> ^{-s} V_h(-T)  W_h(z,\z) + \bigo h 0 (h)\\
& \limt h 0 \pppg z^{-s}  q_1(T,z,\z) \h(p(z,\z))  q_1(T,\vf^T(z,\z)) \pppg {\bar x (T,z,\z)} ^{-s}
\end{aligned}
\]
This proves:
\[
\pppg z ^{-s} q_1(T,z,\z) q_1(T,\vf^T(z,\z)) \pppg {\bar x (T,z,\z)} ^{-s} \leq c T^{-\d}
\]
where $c$ does not depend on $T$, but $\bar x(T,z,\z)$ stays in a bounded subset of $\R^n$, so we must have:
\[
q_1(T,z,\z) q_1(T,\vf^T(z,\z)) \limt T {+\infty} 0
\]
which, by definition of $q_1$, cannot be true unless the classical trajectory starting from $(z,\z)$ goes through $\Oc$ (see \eqref{def-q1}).
\end{proof}

%
%
%

%
%
%
%

\section{Uniform resolvent estimates in Besov spaces}

In order to obtain in Besov spaces the resolvent estimates we proved in weighted spaces, we need another resolvent estimate (see proposition \ref{autre-estim}). We begin with a lemma which turns properties on $\gzle = (H_h-i\e P_h \Th_h^V P_h -z)\inv$ (see section 2) into properties on $(H_h -i\e \Th_h^V -z)\inv$:

\begin{lemma}
With assumptions and notations of theorem \ref{th-mourre}, for all $h,\e \in ]0,1]$ and $z\in\C_{J,+}$ the operator $(H_h-i\e \Th_h^V -z)$ has a bounded inverse (denoted by $\kzle$) which satisfies the following estimates:
\begin{eqnarray}
\label{estim-g1-a} \nr {\kzle} +\nr {\hol \kzle} &\leq& \frac c {\a_h \e}\\
\label{estim-g1-b} \nr{\kzle \alinv} +\nr{\hol \kzle \alinv} &\leq& \frac c {\a_h \sqrt \e}\\
\label{estim-g1-c} \nr{\sqrt{V_h} \kzle} & \leq & \frac c {\sqrt{\a_h} \sqrt\e}\\
\label{estim-g1-d} \nr{\sqrt{V_h} \kzle \alinv
} &\leq& c
\end{eqnarray}
where $c$ is independant of $\e,h\in]0,1]$ and $z \in \C_{I,+}$ for some closed subinterval $I$ of $J$.
\end{lemma}

\begin{proof}
We keep all the notations of the proof of theorem \ref{th-mourre}, in particular $P_h = \vf(\hol)$, $P'_h = 1 - P_h$, $\gzle = (H_h - i\e P_h \Th_h^V P_h)\inv$,\dots \, Applying proposition \ref{propII.5} with $B = \sqrt {\a_h} \sqrt \e P_h$ and $Q = P_h$ gives:
\[
\nr{P_h \gzle P_h} \leq \frac 1 { {\a_h} \e} 
\]
Calculations \eqref{calcul-pp}-\eqref{primegzerau} with $\rle$ replaced by $P_h$ and $P'_h$ show:
\begin{equation} \label{estim1}
\nr{P'_h \gzle P_h} \leq \frac c {\sqrt {\a_h} \sqrt \e} , \quad \nr{P'_h \gzle P'_h} \leq c
\end{equation}
We also have $\nr{P_h \gzle P'_h} \leq \frac c {\sqrt {\a_h} \sqrt \e}$ and hence:
\begin{equation} \label{estim2}
\nr \gzle  + \nr{\hol \gzle}  \leq \frac c {\a_h \e}
\end{equation}
Now three applications of proposition \ref{propII.5} with $B = \sqrt{V_h}$ give:
\begin{equation} \label{estim3}
\nr{\sqrt {V_h} \gzle \pppg {A_h} \inv} + \nr{\sqrt {V_h} \gzle P'_h} \leq c , \quad  \nr{\sqrt {V_h} \gzle} \leq \frac c  {\sqrt {\a_h} \sqrt \e}
\end{equation}
Then, as in \cite{jensenmp84}, we prove that:
\[
\gzlue = \gzlde + i\e \gzlde P'_h (1 -i \e \Th_h^V P_h \gzlde P'_h ) \inv  \Th_h^V P_h \gzlde
\]
is well-defined for $\e$ small enough and is a bounded inverse of $(H_h - i\e \Th_h^V P_h)$ which satisfies estimates \eqref{estim1}-\eqref{estim3} as $\gzle$. Then it remains to define:
\[
\kzle = \gzlue +i \e \gzlue \Th_h^V (1 -i \e{P'_h \gzlue  \Th_h^V})\inv P'_h \gzlue
\]
for $\e$ small enough and check the conclusions of the lemma.
\end{proof}

\begin{proposition} \label{autre-estim}
Let $s > 1$ and $I$ a closed subinterval of $J$. Then there exists $c \geq 0$ such that for all $z \in \C_{I,+}$ and $h \in ]0,1]$:
\begin{equation}
\nr{\1 {\R_-} (A_h) (H_h -z) \inv \alinvs} \leq \frac c {\a_h}
\end{equation}
\end{proposition}

\begin{proof}
We follow the proof of theorem 2.3 in \cite{jensenmp84}. Let 
\[
\tfzle = \1 {\R_-}(A_h) \exp(\e A_h) \kzle \alinvs 
\]
By \eqref{estim-g1-b} we already know that: 
\begin{equation} \label{estim-tfzle}
\nr\tfzle \leq \frac c {\a_h \sqrt \e}
\end{equation}
Then we compute in the sense of quadratic forms on $\Dom_H \cap \Dom _A$:
\[
\begin{aligned}
\frac d {d\e} \tfzle 
& =  \1 {\R_-}(A_h) e^{\e A_h} A_h \kzle \alinvs \\
& \qquad +i \1 {\R_-}(A_h) e^{\e A_h} \kzle (C_V V_h + i[H_h,A_h]) \kzle \alinvs \\
& = \1 {\R_-}(A_h) e^{\e A_h} \kzle A_h \alinvs\\
& \qquad + iC_V \1 {\R_-}(A_h) e^{\e A_h} \kzle V_h \kzle \alinvs \\ 
& \qquad - i \e \1 {\R_-}(A_h) e^{\e A_h} \kzle [\Th_h^V,A] \kzle \alinvs
\end{aligned}
\]
We use complex interpolation to estimate the first term:
\begin{eqnarray*}
\lefteqn{\nr{\1 {\R_-}(A_h) e^{\e A_h} \kzle \pppg {A_h}^{1-s}}}\\
&& \leq \nr{\1 {\R_-}(A_h) e^{\e A_h} \kzle  \alinvs} ^{1 - \frac 1 s} \nr{\1 {\R_-}(A_h) e^{\e A_h} \kzle}^{\frac 1s}\\
&& \leq   c \, { \a _h ^{-\frac 1s} \, \e ^{-\frac 1 {s}}}\nr \tfzle ^{1 - \frac 1s} 
\end{eqnarray*}
For the second term we write:
\[
\begin{aligned}
\nr{\h_-(A_h) e^{\e A_h} \kzle V_h \kzle \alinvs}
& \leq \nr{ \kzle \sqrt{V_h}} \nr{\sqrt{V_h} \kzle \alinvs}\\
& \leq \frac c {\a_h \sqrt\e} 
\end{aligned}
\]
and finally, by assumption (d) and \eqref{estim-g1-a}-\eqref{estim-g1-b}:
\[
\begin{aligned}
\e \nr {\kzle [\Th_h^V,A] \kzle \alinvs}
& \leq c\, \e \a_h \nr {\kzle   }_{\G_h}  \nr{ \kzle \alinvs}_{\G_h}\\
& \leq \frac c {\a_h \sqrt \e }
\end{aligned}
\]
This gives:
\[
\nr{ \frac d {d\e} \a_h \tfzle} \leq  c  \left(\e ^{- \frac 1 s} \nr {\a_h \tfzle} ^{1 - \frac 1s} + \e ^{-\frac 12} \right)
\]
which, together with \eqref{estim-tfzle}, gives the result.
\end{proof}

Let $\O_0 = ]-1,1[$ and $\O_j = \singl{\l \in \R \tqe 2^{j-1}\leq \abs \l < 2^j}$ for $j\in\N^*$. For a selfadjoint operator $F$ on $\Hc$ and $s \geq 0$, the abstract Besov space $B_s(F)$ is defined by:
\[
B_s(F) = \singl{u \in \Hc \tqe \nr u _{B_s(F)}< \infty} 
\]
where:
\[
\nr u _{B_s(F)} = \sum_{j\in\N} 2^{js} \nr{ \1 {\O_j}(F) u}_\Hc
\]
The norm of its dual space $B_s^*(F)$ with respect to the scalar product on $\Hc$ is:
\[
\nr v _{B_s^*} = \sup_{j\in\N} 2^{-js} \nr {\1 {\O_j}(F) v}_\Hc
\]
When $F$ is the multiplication by $x$ on $L^2(\R^n)$ we recover the usual Besov spaces $B_s$ and $B_s^*$ and the norm we have just defined for $B_s^*$ is equivalent to the usual one:
\[
\sup_{R \geq 1} R ^{-s} \left( \int _{\abs x < R} \abs {v(x)}^2\, dx \right)^{\frac 12}
\]

\begin{theorem} \label{th-besov}
Let $(H_h)$ be an abstract family of dissipative operators as in section 2, $(A_h)$ a conjugate family for $(H_h)$ on $J$ with bounds $(\a_h)$ as in definition \ref{def-conj} and $s \geq \frac 12$. Then for all closed subinterval $I$ of $J$ there exists $c \geq 0$ such that for any $z \in \C_{I,h}$ and $h \in ]0,1]$ we have:
\[
\nr{(H_h-z)\inv } _{B_s(A_h) \to B_s^* (A_h)} \leq \frac c {\a_h}
\]
\end{theorem}

Now that we have theorem \ref{th-mourre} and proposition \ref{autre-estim}, we can follow word by word the proof of the analog theorem for selfadjoint operators (see theorem 2.2 in \cite{wang07}). Applied to our dissipative Schrödinger $H_h = -h^2\D + V_1(x) -i\nu(h) V_2(x)$, this gives:

\begin{theorem}
Let $E > 0$ and $s \geq \frac 12$. If all bounded trajectories of energy $E$ meet $\Oc$, then there exists $\e,h_0 >0$ and $c \geq 0$ such that with $J = [E-\e,E + \e]$ we have for all $z \in \C_{J,+}$ and $h \in ]0,h_0]$:
\[
\nr{(H_h-z)\inv}_{B_s \to B_s^*} \leq \frac c {h \tilde \nu(h)}
\]
\end{theorem}
\noindent
(we recall that $\tilde \nu(h) = \min(1 , \nu(h)/h)$).

\begin{proof}
We already have a conjugate family $(\tilde \nu (h) F_h)$ for $(H_h)$. So we only have to apply the abstract theorem \ref{th-besov}, \eqref{devhi}, and the estimate:
\begin{equation}\label{estim-hi-besov}
\nr{(H_h -i)\inv}_{B_s \to B_s(F_h)} \leq c
\end{equation}
with a similar estimate for dual spaces. To prove \eqref{estim-hi-besov}, we use the idea given in \cite[14.1]{hormander2}. For any $u \in B_s(F)$ and $k \in\N$, since the $\1 {\O_j}(F) u$ for $j\in\N$ are pairwise orthogonal we have:
\begin{eqnarray*}
\lefteqn{\nr u _{B_s(F_h)}}\\
&& = \sum_{0\leq j \leq k} 2^{js} \nr{\1{\O_j} (F_h) u} + \sum_{j > k} 2^{-js} \nr{2^{2js}\1{\O_j} (F_h) u} \\
&& \leq \left(\sum_{j\leq k} 2^{2js}\right)^{\frac  12} \left( \sum_{j\leq k} \nr{\1 {\O_j}(F_h)u}^2 \right)^{\frac 12} +\left(\sum_{j > k} 2^{-2js}\right)^{\frac  12} \left( \sum_{j > k} \nr{2^{2js} \1 {\O_j}(F_h)u}^2 \right)^{\frac 12}\\
&& \leq c_s 2^{ks} \nr u + c_s 2^{-ks} \nr {\pppg {F_h} ^{2s} u}
\end{eqnarray*}
and hence, for $\f \in B_s$, using the fact that the operator $\pppg {F_h}^{2s} (H_h-i)\inv \pppg x ^{-2s}$ is bounded in $\Lc(L^2(\R^n))$ uniformly in $h$ we have:
\[
\begin{aligned}
\nr{(H_h-i)\inv \f}_{B_s(F_h)} 
& \leq \sum_{k\in\N} \nr{(H_h -i)\inv \1 {\O_k} (x)\f}_{B_s(F_h)} \\
& \leq c_s \sum_{k\in\N} 2^{ks} \nr{(H_h-i)\inv \1{\O_k}(x) \f} + c_s \sum_{k\in\N} 2^{-ks} \nr{\pppg X^{2s} \1 {\O_k}(x) \f} \\
& \leq c_s \sum_{k\in\N} 2^{ks} \nr{\1{\O_k}(x) \f} + c_s \sum_{k\in\N} 2^{-ks}  \nr{\pppg X^{2s} \1 {\O_k}(x) \f} \\
&\leq c_s \nr \f _{B_s} + c_s \sum_{k\in\N} 2^{ks} \nr{\1{\O_k}(x) \f}\\
&\leq c_s \nr \f _{B_s}
\end{aligned}
\]
\end{proof}

\appendix

\section{Unitary dilations and dissipative Schrödinger operators} \label{sec-dil}

In order to use the selfadjoint theory to study dissipative operators, we have mostly used the assumption that $H$ is a perturbation of its selfadjoint part $H_1$. However, by the theory of unitary dilations, there are other selfadjoint operators we can use:

\begin{definition} \label{def-dil}
Let $T$ be a bounded operator of the Hilbert space $\Hc$. A bounded operator $U$ on a Hilbert space $\Kc$ is said to be a dilation of $T$ if $\Hc \subset \Kc$ and for all $\f,\p \in \Hc$ and $n\in\N$ we have:
\[
\innp{U^n \f}\p _ \Kc = \innp{T^n \f} \p _\Hc
\]
\end{definition}

The theory of unitary dilations for a contraction is developped in the book of B.S.-Nagy and C. Foias (\cite{nagyf}). In particular, we know that every contraction has a unitary dilation. This also holds for semigroups of contractions: if $(T(t))_{t \geq 0}$ is a semigroup of contractions on $\Hc$, there exists a unitary group $(U(t))_{t\in\R}$ on $\Kc \supset \Hc$ such that $U(t)$ is a dilation of $T(t)$ for all $t\geq 0$. Then, if $H$ is a dissipative operator on $\Hc$, there is a unitary group of dilations $(U(t))$ on $\Kc\supset \Hc$ for the semigroup $(e^{- {it} H})$ generated by $H$. The unitary group $(U(t))$ is generated by a selfadjoint operator $K$ on $\Kc$, the properties of which we can use to study the dissipative operator $H$. Note that $K$ is usually said to be a selfadjoint dilation of $H$ but is not a dilation of $H$ in the sense of definition \ref{def-dil}.

Much is said on the abstract theory in \cite{nagyf}, but there is an explicit study of the dissipative Schrödinger operator case in \cite{pavlov77}. In particular an exemple of dilation is given. Here we recall this example in the semiclassical setting:
\begin{proposition}
Let $-h^2\D + V_1 -ihV_2$ be a dissipative Schrödinger operator on $L^2(\R^n)$ as in section \ref{sec-appl-schr}, $W_h = \sqrt {2hV_2}$, $\O = \supp V_2$, $\Kc = L^2(\R_- ,L^2( \O)) \oplus L^2(\R^n) \oplus L^2(\R_+ ,L^2( \O)) $ and $P$ the orthogonal projection of $\Kc$ on $L^2(\R^n)$. Then the operator:
\[
K_h : \f = (\f_-,\f_0,\f_+)  \mapsto \left(-i\f_-', \hun \f_0 - \frac {W_h}2 (\f_-(0) + \f_+(0)), -i\f_+'\right)
\]
with domain:
\[
\Dom(K_h) = \singl{(\f_-,\f_0,\f_+) \tqe \f_\pm \in H^1(\R_\pm,L^2(\O))\text{ and } \f_+(0) - \f_-(0) = iW_h \f_0} \subset \Kc
\]
(where $H^1$ is the Sobolev space of $L^2$-functions with first derivative in $L^2$) is a selfadjoint operator which satisfies:
\begin{eqnarray*}
\forall z \in \C_+ ,&& \restr{P(K_h-z)\inv}{L^2(\R^n)} = (H_h-z)\inv \\ 
\forall z \in \C_+ ,&& \restr{P(K_h-\bar z)\inv}{L^2(\R^n)} = (H_h^*-\bar z)\inv \\ 
\forall t \geq 0 ,&& \restr{Pe^{-\frac {it}h K_h}}{L^2(\R^n)} = e^{-\frac {it}h H_h} \\ 
\forall t \leq 0 ,&& \restr{Pe^{-\frac {it}h K_h}}{L^2(\R^n)} = e^{-\frac {it}h H_h^*} \\ 
\end{eqnarray*}

\end{proposition}

\begin{proof}
The proof of the proposition is straightforward calculations. We first have to check that $K$ is symmetric, that $\Dom(K^*) \subset \Dom(K)$ and then that for $z \in \C_+$ we have $(\p_-,\p_0,\p_+) = (K-z)\inv (\f_-,\f_0,\f_+)$ where:
\begin{eqnarray*}
\p_-(r) &=& i\int_{-\infty}^r e^{iz(r-s)} \f_-(s)\, ds\\
\p_0 &=& (H_h-z)\inv (\f_0 + W_h \p_-(0))\\
\p_+(r) &=& (\p_-(0) + iW_h\p_0) e^{izr} + i \int_0^r e^{iz(r-s)} \f_+(s)\,ds
\end{eqnarray*}
and an analog for $(K-\bar z)\inv$. To prove the last statement, we show that the generator of the semigroup $t\mapsto \restr{Pe^{-\frac {it}h K_h}}{L^2(\R^n)}$ must be $H$ using the result on the resolvent. Details are given in \cite{pavlov77}.
\end{proof}

\section*{Acknowledgments}
I am very grateful to Xue-Ping Wang for many helpful discussions and stimulating questions about the subject.

\bibliographystyle{amsalpha}
\bibliography{bibliothese}

\end{document}